\documentclass[11pt]{amsart}
\usepackage{amscd,amssymb,amsmath,enumerate}
\usepackage{amsmath,amsthm}
\usepackage[all]{xy}
\usepackage{graphicx}
\usepackage{graphics}
\usepackage{latexsym}
\usepackage[all]{xy}
\usepackage{psfrag}
\xyoption{matrix} \xyoption{arrow}
\usepackage{parskip}
\usepackage[normalem]{ulem}
\usepackage{hyperref}

\newtheorem{proposition}{Proposition}[section]
\newtheorem{theorem}[proposition]{Theorem}
\newtheorem{lemma}[proposition]{Lemma}
\newtheorem{corollary}[proposition]{Corollary}

\newtheorem{definition}[proposition]{{Definition}}
\newenvironment{defn}{\begin{definition} \rm}{\end{definition}}

\newtheorem{remark}[proposition]{{Remark}}

\newtheorem{Example}[proposition]{Example}

\newcommand{\cA}{{\mathcal A}}
\newcommand{\cB}{{\mathcal B}}

\newcommand{\cI}{{\mathcal I}}
\newcommand{\cN}{{\mathcal N}}
\newcommand{\cF}{{\mathcal F}}

\newcommand{\cP}{{\mathcal P}}
\newcommand{\cQ}{{Q}}

\newcommand{\cH}{{\mathcal H}}
\newcommand{\cT}{{\mathcal T}}
\newcommand{\cG}{{\mathcal G}}

\newcommand{\cS}{{\mathcal S}}
\newcommand{\cV}{{\mathcal V}}

\renewcommand{\dim}{\operatorname{dim}\nolimits}
\newcommand{\gldim}{\operatorname{gl.dim}\nolimits}

\newcommand{\grb}{Gr\"obner\ }
\newcommand{\tip}{\operatorname{tip}}

\newcommand{\Span}{\operatorname{Span}}
\newcommand{\Imo}{I_{Mon}}
\newcommand{\Alge}{\operatorname{Alg}\nolimits}

\newcommand{\Hilb}{\operatorname{\text{Hilb}}\nolimits}

\newcommand{\Alg}{\operatorname{Alg}\nolimits}

\usepackage{color}
\definecolor{candyapplered}{rgb}{1.0, 0.03, 0.0}
\makeatletter
\def\thm@space@setup{%
  \thm@preskip=0.7cM \thm@postskip=0.3cM
}
\makeatother

\begin{document}

\date{today}
\title[Algebras and varieties]
{Algebras and varieties}

\author[Green]{Edward L.\ Green}
\address{Edward L.\ Green, Department of
Mathematics\\ Virginia Tech\\ Blacksburg, VA 24061\\
USA}
\email{green@math.vt.edu}

\author[Hille]{Lutz Hille}
\address{Lutz Hille\\
Department of Mathematics \\
University of Muenster \\
}
\email{lutzhille@uni-muenster.de}

\author[Schroll]{Sibylle Schroll}
\address{Sibylle Schroll\\
Department of Mathematics \\
University of Leicester \\
University Road  \\
Leicester LE1 7RH, UK
}
\email{schroll@le.ac.uk}

\thanks{The first and third author were partially supported by an LMS scheme 4 grant. The third author is supported by the EPSRC through the Early Career Fellowship EP/P016294/1}

\begin{abstract}
In this paper we introduce new affine algebraic varieties whose points correspond to  associative algebras.  We show that the algebras within a variety share many important homological properties. In particular, any two algebras in the same variety have the same dimension. The case of finite dimensional algebras as well as that of graded algebras arise as subvarieties of the varieties we define.  As an application we show that for algebras of global dimension two over the complex numbers, any algebra in the variety continuously deforms to a monomial algebra.  
\end{abstract}
\date{\today}

\maketitle

\tableofcontents

\section{Introduction}

The interplay between commutative algebra and algebraic geometry   plays a fundamental role in these areas, see for example \cite{Eisenbud}. The use of algebraic geometry in other areas of mathematics has led to important results. There are numerous examples of this, including:   in the area of  non-commutative geometry, the classification of Artin Schelter regular algebras of global dimension three and their modules \cite{AS, Artin Tate vdB1, Artin Tate vdB2};  in the context of resolving singularities using non-commutative crepant resolutions, picking some examples from the extensive literature in the area, we cite \cite{ Buchweitz Leuschke vdB, Iyama Wemyss, SvdB, vdB};  in the area of weighted projective lines, we mention, for example,  
\cite{Kussin Lenzing Meltzer1, Kussin Lenzing Meltzer2}
  and the recent field of tropical geometry, for example see \cite{Sturmfels book}.

In this paper, we establish a new connection between (non-)commutative algebras and affine algebraic varieties. 

  A powerful  way of studying non-commutative algebras is to study 
quotients of path algebras of quivers. Quotients of paths algebras encompass many naturally arising classes of algebras. For example, free associative algebras are particular cases of path algebras, and hence every finitely generated algebra is the quotient of a path algebra.
Moreover,  every finite dimensional algebra over an algebraically closed field  is Morita equivalent to a quotient of the path algebra of a quiver, where the quiver is an invariant of the algebra  up to isomorphism.

In commutative polynomial rings, monomial ideals are an important special class of ideals. In a similar fashion, ideals generated by paths (as opposed to linear combinations of paths) in a path algebra play an analogous role. In fact, an ideal in a path algebra generated by paths is called a {\it monomial ideal} and the quotient by such an ideal is called a {\it monomial algebra}.

Commutative finitely generated algebras are quotients of polynomial rings  by ideals of polynomials, where the zero loci of these polynomials give rise to the associated affine geometry. The theory of Gr\"obner bases is a central tool in the study of ideals in polynomial rings. 

In this paper, we use non-commutative \grb basis  theory to construct new  affine algebraic varieties whose points are in one-to-one correspondence with quotients of path algebras. We note that there is no assumption that the
quotients of  paths are finite dimensional and our results include both finite and infinite dimensional algebras.

The precise construction of the new affine algebraic varieties is in  Section~\ref{Construction of varieties}. The key property of such an affine algebraic variety $\cV$ is  that the following  hold
\begin{enumerate}
\item  Every point in $\cV$ corresponds to an algebra.
\item There exists a distinguished point $A_{Mon}$  in $\cV$  corresponding to a monomial algebra.
\item The characteristics of $A_{Mon}$ govern the characteristics of all the algebras corresponding to the points in $\cV$. 
\end{enumerate}

Furthermore, every quotient of a path algebra by a finitely generated ideal is in at least one such variety. Moreover, fixing a presentation and an order on the paths, every quotient of a path algebra is in a unique variety.  Surprisingly,  we are able to show in general, that the set of algebras having the same associated monomial algebra  is a variety.

 In the classical theory, loosely speaking,  one has an algebraic group acting on  algebras of a fixed dimension (vector),
and one studies the algebras modulo the group action.   Although the algebras associated to
points in one of our varieties all have the same dimension   there can be many different
varieties where the algebras  corresponding to points have the same dimension vector.  But there
seems to be no group action on our varieties.  For our varieties, the complexity of the defining ideal is directly
related to 3 factors:
\begin{enumerate}
\item the complexity of the set of parallel paths
\item a chosen order $\succ$ on paths, and 
\item the complexity of constructing a \grb basis.
\end{enumerate}

We point out that the most difficult, and probably most interesting case is the case of local algebras, since in the corresponding paths algebras all nonzero paths are parallel. In general,  one should not expect a direct connection between the theory presented
here and the classical theory.   The results in Section \ref{sec:gldim2} provides an example of how using our
results, we can obtain new insights.

 Given a variety of algebras, we describe the construction of certain subvarieties. We show that varieties corresponding to finite dimensional algebras arise as such subvarieties. We also show that varieties associated to graded algebras correspond to such subvarieties.

The general idea of this paper is to exploit the extensive knowledge of monomial algebras. Monomial algebras have been intensively studied because they are much more amenable to computations of their ideal structure and homological properties than general algebras. For example, there is an algorithmic approach to find the Cartan matrix and the necessary and sufficient conditions for the algebra to be quasi-hereditary \cite{GSch}; there is an algorithm to find the graded minimal projective resolution of a monomial algebra and hence the global dimension \cite{GHZ}; every quadratic monomial algebra is  Koszul \cite{GH}.

  The following Theorem summarises the main common properties of all algebras in a variety many of which allow to make use of the well-understood knowlegde of these properties for monomial algebras.

 \begin{theorem}\label{thm intro}
Suppose that $\cV$ is an affine algebraic variety with distinguished point $A_{Mon}$ as described in (1)-(3) above. Then 
 \begin{enumerate}
 \item All algebras in $\cV$ have the same dimension equal to the dimension of $A_{Mon}$.
 \item All finite dimensional algebras in $\cV$ have the same Cartan matrix.
 \item gldim $(\Lambda) \leq $ gldim $(A_{Mon})$, for all algebras $\Lambda$ in $\cV$.
 \item If gldim $(A_{Mon}) < \infty$  then the Cartan determinant conjecture holds for all (finite dimensional) algebras in $\cV$.
\item  If $A_{Mon}$  is Koszul  then all graded quadratic algebras in $\cV$ are Koszul. The analogous statement holds in the case of $N$-Koszul algebras, $N \geq 3 $.
 \item  If $\Lambda \in \cV$ is a positively $\mathbb Z$-graded algebra then $A_{Mon}$ is positively $\mathbb Z$-graded and the Hilbert series of $\Lambda$ and $A_{Mon}$ are equal.
 \item If $A_{Mon}$ is quasi-hereditary then all algebras in $\cV$ are quasi-hereditary.

 \end{enumerate}
 \end{theorem}

The Cartan determinant conjecture states that the determinant of the Cartan matrix of a finite dimensional algebra of finite global dimension is equal to 1. This conjecture has been shown to hold for monomial algebras \cite{W}. 
A consequence of the above theorem is that if the Cartan determinant of a finite dimensional algebra of finite global dimension is equal to 1, then the Cartan determinant conjecture holds for all algebras  in that variety. Hence to prove the Cartan determinant conjecture one need only study algebras of finite global dimension whose associated monomial algebra is of infinite global dimension for any order.

\section{Gr\"obner bases}

In this Section and Section~\ref{sect-reduced}, we give a brief survey of non-commutative \grb basis theory for path algebras following \cite{G Intro Groebner}. 

We fix the following notation throughout the paper.  Let $K$ be any field. If $X$ is  an infinite set of vectors in a $K$-vector space, then
$\sum_{x\in X}c_xx$ with $c_x\in K$ has as an unstated assumption that all but a finite
number of $c_x$ equal $0$. 

Let $Q$ be a finite quiver, and let $\cB$ be the set of finite (directed) paths in $Q$. 
 Note that we view the vertices of $Q$ as paths of length zero and so they are elements in $\cB$. Furthermore, elements in the path algebra $KQ$ are read left to right.

We say that a nonzero element $x \in KQ$ is \emph{uniform} if there exists vertices $v, w \in Q_0$ such that $vxw=x$ and we say that $v$ is the \emph{origin vertex} of $x$,  denoted by $\mathfrak o(x)$, and
$w$ is the \emph{terminus vertex} of $x$, denoted by $\mathfrak t(x)$.

To define a  Gr\"obner basis theory,  we need to choose an admissible order on $\cB$. 

\begin{defn}
An \emph{admissible order} $\succ$ on $\cB$ is a well-order on $\cB$ satisfying the following conditions 
\begin{enumerate}
\item If $p, q, r \in \cB$ and $p \succ q$ then $pr \succ qr$ if $pr$ and $qr$ are nonzero. 
\item If $p, q, s \in \cB$ and $p \succ q$ then $sp \succ sq$ if $sp$ and $sq$ are nonzero. 
\item  If $p,q,r,s  \in \cB$ with $p = rqs$ then $p  \succeq q$ if $p$ is nonzero. 
\end{enumerate}
\end{defn}

Recall that a well-order is a total order such that every non-empty subset has a minimal element. 

We fix an admissible order $\succ$  on $\cB$; { for example, $\succ$ can be the length-lexicographical order. 
Using the order $\succ$ we define what the tip of an element in $KQ$ is and what the tip of a subset of $KQ$ is. 

\begin{defn} If $ x = \sum_{ p \in \cB} \alpha_p p$, with $\alpha_p\in K$, almost all $\alpha_p=0$,  and $x \neq 0$ then define the {\it tip of $x$} to be

\begin{center} $\tip (x) = p$ if $\alpha_p\ne 0$ and $p \succeq  q$ for all $q$ with $ \alpha_q\ne 0$. \end{center}

If $X \subseteq KQ$ then $$\tip(X) = \{ \tip(x) | x \in X \setminus \{0\} \}.$$
\end{defn}

 \begin{defn}
 We say $\cG$ is a \emph{Gr\"obner basis} for $I$ (with respect to $\succ$) if  $\cG$ is a set of uniform elements in $I$  such that $$\langle \tip (I) \rangle = \langle \tip (\cG) \rangle $$ where $\langle A \rangle$ denotes the ideal generated by $A$. 
 \end{defn}
 
 Equivalently a set $\cG \subset I$ is a \grb basis for $I$ with respect to $\succ$ if for every $x \in I$, $ x \neq 0$, there exists a $g \in \cG$ such that $\tip (g)$ is a subpath of $\tip (x)$  where if $p,q\in \cB$, we say $p$ is a \emph{subpath} of $q$ and write $p| q$, if there exist $r,s\in\cB$ such that
$q=rps$.
 
Although not immediately obvious, it is easy to prove that

\begin{proposition}\cite{G Intro Groebner} If $\cG$ is a \grb basis for $I$, then $\cG$ is a generating set for $I$ that is
$$ \langle \cG \rangle = I.$$
\end{proposition}

 We call elements of $\cB$ \emph{monomials} and we say an ideal in $KQ$ is a \emph{monomial ideal} if it can be generated by monomials. We recall the following well-known facts about monomial ideals. 
 
 \begin{proposition}\cite{G Intro Groebner}\label{prop-mono} Let $I$ be a monomial ideal in $KQ$. Then 
 
 (1) there is a unique minimal set $\cT$ of monomial generators for $I$  independent of the chosen admissible order; 

(2) $\cT$ is a \grb basis for $I$ for any admissible order on $\cB$;

 (3)  if $x = \sum_{ p \in \cB} \alpha_p p$, with $\alpha_p \in K$,  then $x \in I$ if and only if $p \in I$ for all $\alpha_p \neq 0$. 
 \end{proposition}
 
 Let $\cN = \cB \setminus \tip (I)$ and let $\cT$ be the minimal set of monomials that generate the 
monomial ideal {$\langle \tip (I) \rangle$.} Note that $ \cN = \{ p \in \cB \mid t \mbox{ is not a subpath of $p$, for all $t \in \cT$}\}$.
We call  $\cN$ the set of \emph{nontips} of $I$. 

We note that if $KQ/I$ is finite dimensional, by \cite{G Intro Groebner} there exists a finite \grb basis and it follows that $\cT$ is a finite set. 
 
 The following result is of central importance and we include a proof for completeness.
 
 \begin{lemma} \rm{ \bf{(Fundamental Lemma)}}
 As $K$-vector spaces 
 $$ KQ = I \oplus \mbox{ Span}_K \cN.$$
 \end{lemma}
 
 \begin{proof}   First we note that $I\cap \Span_K \cN = (0)$; since  if
$x\in I\cap \Span_K \cN$ and $x\ne 0$, then $x\in I$ implies that
$\tip(x) \in \tip(I)$, contradicting that $\tip (x) \in \Span_K \cN $.

Now suppose that $I + \Span_K \cN \ne KQ$.  Choose $x\in KQ$ with
minimal tip such that $x\notin I +\Span_ K \cN $. Let $t=\tip(x)$ and
note that $t\in \tip(I)$ or $t\in\cN$.   Let $\alpha_t$ be the coefficient of $t$
in $x=\sum_{p\in\cB}\alpha_pp$ where $\alpha_p\in K$.
If $t\in \tip(I)$, then there is some $y \in I$ such that $\tip (y)=t$.
If $\beta$ is the coefficient of $\tip(y)$ in $ y$, then 
$ \tip(\beta x-(\alpha_ty))$ is either $0$ or has smaller tip than $t$.
By our assumption of the minimality of the tip of $x$, we conclude
that $\beta x-(\alpha_t)y \in I+\Span_K \cN $.   Since
$y\in I$, we see that $x\in I+ \Span_K \cN $, a contradiction.

On the other hand, if $t\in\cN$, $x-\alpha_t t$ is $0$ or has smaller tip
than $x$.  By   minimality of $t$, $x-\alpha_tt$ is in $ I+ \Span_K \cN $ 
and hence $x\in I+\Span_K \cN $, a contradiction.
 \end{proof}
 
It is an immediate consequence of the Fundamental Lemma that if $x \in KQ \setminus \{ 0\}$, then $x = i_x + N_x$ for a unique $i_x \in I$ and a unique $N_x \in \mbox{ Span}_K \cN$.
 
  \begin{corollary}\label{Basis N} As $K$-vector spaces, $ \mbox{ Span}_K \cN$ is isomorphic to $KQ/I$ and hence $\cN$ can be identified with  a $K$-basis of $KQ/I$.
 \end{corollary}

Therefore we can identify $KQ/I$ with $\mbox{ Span}_K \cN$ where for $x, y \in \mbox{ Span}_K \cN$, the multiplication of $x$ and $y$ in $KQ/I$ equals $N_{x \cdot y}$ where  $x \cdot y $ is the usual multiplication in $KQ$.

\begin{defn}\label{Definition reduction}
Let $\cF$ be a set of nonzero  elements in $KQ$ and $x=\sum_{p\in\cB}\alpha_pp\ne 0$ be an element
of $KQ$. 

\noindent
1) {\it Simple reduction of $x$ by $\cF$:}    Suppose for some $p$ with $\alpha_p\ne 0$ there is some $f\in\cF$ such that
$\tip(f)|p$.  If the coefficient of $\tip(f)=\beta$ and $p=r\tip(f)s$ with $r,s\in\cB$, then a \emph{simple reduction
of $x$ by $\cF$}, denoted by $x\to_{\cF}y$, is  $y=\beta x-\alpha_prfs$.  Note that this 
simple reduction replaces $\alpha_pp$ in $x$ by a linear combination of paths smaller than $p$.

\noindent
2) {\it Complete reduction of $x$ by $\cF$:}
A \emph{complete reduction of $x$ by $\cF$}, denoted by $x\Rightarrow_{\cF}y_n$, is a sequence of simple reductions
$(\cdots ((x\to_{\cF}y_1)\to_{\cF}y_2)\to_{\cF}\cdots)\to_{\cF}y_n$, such that either $y_n=0$ or
$y_n$ has no simple reductions by $\cF$. 

\noindent 3) We say a set $X\subseteq KQ$ is \emph{tip-reduced} if for all $x,y\in X\setminus\{0\}$
with $x\ne y$, $\tip(x)$ is not a subpath of $\tip(y)$. 
\end{defn}

\vskip .2in

Since $\succ$ is a well-order on $\cB$, if $x\in KQ$ and $\cF$ is a set of nonzero elements in $KQ$, then there exists $y\in KQ$ such that
$x\Rightarrow_{\cF}y$.   Note that $x$ may be completely reduced by $\cF$ to two distinct elements $y$ and $y'$ by different sequences of simple reductions. However,   by the Fundamental Lemma, if $\cG$ is a \grb basis for $I$,
given $x\in KQ$ there is a unique $N_x \in \mbox{ Span}_K \cN$ such that $x\Rightarrow_{\cG}N_x$.

The next result is an important property of reduction and is an immediate consequence
of the Fundamental Lemma.

\begin{proposition}\label{prop-red2zero} Let $\cG$ be a \grb basis for an ideal
$I$ and $x$ a nonzero element of $ KQ$.  Then $x\in I$ if and only if
$x\Rightarrow_{\cG}0$.
\end{proposition}

 We have the following result whose 
proof is left to the reader.

\begin{lemma}\label{lemma-simplered}
Let $X =\{x_1,\dots, x_m\}\subseteq KQ$ be a finite set of uniform elements and set $I=\langle X
\rangle$.  If for some $i$, we have $x_i\to_{X\setminus\{x_i\}}y$ then $I$ can be generated
by $\{x_1,\dots, x_{i-1},y, x_{i+1},\dots,x_r\}$.

\end{lemma}

The following Corollary follows from the above lemma and the well-ordered property of
$\cB$.   We leave the details to the reader. 

\begin{corollary}\label{cor-tipreduce}Let $X =\{x_1,\dots, x_m\}\subseteq KQ$ be a finite set of 
uniform elements and $I=\langle X \rangle$.
After a finite sequence of simple reductions of the elements of $X$ as given in Lemma
\ref{lemma-simplered}, one may obtain a finite, tip-reduced set $ Y$ such that $Y $ generates
$I$.

\end{corollary}

The above Corollary implies that if we have a finite set of uniform generators of an ideal
$I$ in $KQ$, without loss of generality, we may  assume that the set of uniform generators
is tip-reduced.  

The last concept needed is that of an overlap relation. 

\begin{defn}\label{def-overlap} Let $x=\sum_{p\in\cB}\alpha_pp,$\
$y=\sum_{q\in\cB}\beta_qq\in KQ$.  Suppose that $t=\tip(x)$, $t'=\tip(y)$
and $tm=nt'$ for some $m,n\in \cB$ where the length of $n$ is at least  one and the length of
$m$ is at least one and strictly less than the length of $t'$. 
  Then the \emph{overlap relation},
$\mathbf o(x,y,m,n)$, is
\[\mathbf o(x,y,m,n)=(\beta_{t'})xm-(\alpha_t)ny.\] 
\end{defn}

We now state the result that allows one to check if a set of tip-reduced uniform generators of
an ideal $I$
is a \grb basis for $I$.  Furthermore, if the set is not a \grb basis, it allows one to construct
a \grb basis (in possibly an infinite number of steps).

\begin{theorem}\label{thm-buch} \cite{B, G Intro Groebner, Mora} Let $K$ be a field, $Q$ a finite quiver, and $\succ$ an
admissible order on the set $\cB$ of finite directed paths in $Q$.  Suppose that $\cG$ is
a (not necessarily finite) tip-reduced set of uniform generators of an ideal $I$.  Then $\cG$ is a
\grb basis for $I$ with respect to $\succ$ if and only if every overlap relation of any two elements
of $\cG$ completely reduces to $0$ by $\cG$.
\end{theorem} 
 
\section{The reduced \grb basis and the associated monomial algebra}\label{sect-reduced}
In this section we fix a field $K$, a quiver $Q$, an ideal $I$ in $KQ$, and an admissible order $\succ$
on $\cB$.

\begin{defn}
Let $I_{Mon}$ be the ideal in $KQ$ generated by $\tip(I)$ and define the \emph{associated monomial algebra of $\Lambda=KQ/I$} to be $\Lambda_{Mon} =KQ / I_{Mon}$. 
\end{defn} 
Since $I_{Mon}$ is a monomial ideal, by Proposition \ref{prop-mono}(1), there is a unique
minimal set $\cT$ of paths that generate $I_{Mon}$.  Recalling that $\cN=\cB\setminus \tip(I)$,
by the Fundamental Lemma, if $t\in\cT$,
then there exist unique elements $g_t\in I$ and $N_t\in \mbox{ Span}_K \cN$, such that
$t=g_t+N_t$.  Set $\cG=\{g_t\mid t\in\cT \} \subset I$. Since $t$ is uniform, $g_t$  and $N_t$ are uniform. Note that
$\tip(g_t)=t$ since $N_t\in \mbox{ Span}_K \cN$.  Thus, $\tip(\cG)=\cT$ and hence
$\cG$ is a \grb basis for $I$ since $\langle \tip(\cG)\rangle=\langle \cT\rangle =\langle\Imo\rangle=\langle
\tip(I)\rangle$.

\begin{defn}  The set $\cG=\{g_t\mid t\in\cT \} \subset I$ defined above is called the \emph{reduced
\grb basis for $I$ (with respect to $\succ$)}.
\end{defn}

The next result lists some basic facts about reduced \grb bases and associated
monomial algebras.

\begin{proposition}\label{prop-redgb} Let $I$ be an ideal in the path algebra $KQ$
and let $\Lambda=KQ/I$.  Let $\cT$ be the unique minimal set of monomials generating
$\langle \tip(I)\rangle$ and let $\cG$ be the reduced \grb basis for $I$. Then the following hold.
\begin{enumerate}
\item  The reduced \grb basis for $\Imo$ is $\cT$.
\item  $\dim_K(\Lambda)=\dim_K(\Lambda_{Mon}) = \vert \cN \vert $.
\item   We have that $\Lambda$ is finite dimensional if and only if $\cN$ is a finite set.
\item If $\cN$ is finite, then $\cT$ is finite.
\end{enumerate}
\end{proposition}

Note that (4) holds since $\cT = \{ a_1 \ldots a_m  \in \cB \mid a_1 \ldots a_{m-1}, 
a_2 \ldots a_m \in  \cN \}$. Furthermore, the converse of (4) is false in general. 

Keeping the notation above, we write elements of both $\Lambda$ and $\Lambda_{Mon}$
as $K$-linear combinations of elements in $\cN$.   The difference is how these elements
multiply in $\Lambda$ and $\Lambda_{Mon}$.  Next  assume that $I$ is an admissible ideal.  Notice that $I$  admissible implies that $I_{Mon}$ is admissible. 
 If $I$ is admissible,  then the set of vertices and arrows
are always in $\cN$, and both $\cT$ and $\cN$ are finite sets.  

In the sections that follow we will apply \grb basis theory to the case where the field is of
the form $L=K(x_1,\ldots,x_n)$, the quotient field of the (commutative) polynomial ring in $n$ variables over
$K$.  Let $R=K[x_1,\ldots, x_n]$. Given a quiver $Q$ we  consider elements of $LQ$ and the subset  of elements
 whose coefficients lie in $R$.  We say $r=\sum_{p\in\cB}\alpha_pp\in LQ$ is in $RQ$ if each $\alpha_p\in R$.

\begin{proposition}\label{prop-poly}
Let  $\cF$ be a set  of tip-reduced uniform elements in $RQ\subset LQ$.
Then \begin{enumerate}
\item If $r\in RQ$ and $r\Rightarrow_{\cF}s$ then $s\in RQ$.
\item If $f, f'\in\cF$ with $\tip(f)p=q\tip(f')$ with $1\le \ell(p) <\ell(\tip(f'))$ then the overlap
relation ${\bf o}(f,f',p,q)$ is in $RQ$.
\end{enumerate}
\end{proposition}

\begin{proof} From the Definition \ref{Definition reduction}\ 1) we see that if $r\in RQ$ and
$r\to_{\cF}r'$, then $r'\in RQ$. From Definition \ref{Definition reduction} 2)  part 1)
follows.   Part 2) follows from Definition \ref{def-overlap} and that $\cF\subset RQ$.
\end{proof}

\section{ The varieties }
\label{Construction of varieties}

In this section we introduce varieties $\cV_{\cT}$ whose points are in 
one-to-one correspondence with a class of algebras. 

We note that throughout the paper, there are two different starting points one can take. On the one side, starting with an algebra $KQ/I$  where $KQ$ is the path algebra of a quiver $Q$ and $I$ is an ideal in $KQ$, we show how to construct a variety in which $KQ/I$ corresponds to a point. On the other side, we start with a quiver $Q$ and a set of paths $\cT$ in the quiver. We construct the variety whose distinguished point corresponds to the monomial algebra  given by the quotient of $KQ$ by the ideal generated by the set of paths $\cT$.

Given a quiver $Q$, the definition of the variety is dependent on the choice of
 an admissible order $\succ$ on $\cB$ and on a tip-reduced set $\cT \subset \cB$ of paths of
length at least 2.

We begin by defining a set of algebras.

\begin{defn} Let $\cT \subset \cB$ be a tip-reduced set of elements of length at least two and let $\succ$ be an order on $\cB $. 
 Define
$$\Alge_{\cT} = \{ KQ/ I \mid   I \mbox{ ideal of }  KQ \mbox{ and } 
\Imo=\langle \cT \rangle \}.$$
\end{defn}

If  $x$ and $y$ are
uniform elements of $KQ$, then
we say $x$ is \emph{parallel to $y$} and write $x\| y$ if $\mathfrak o(x)=\mathfrak o(y)$
and $\mathfrak t(x)=\mathfrak t(y)$.

Let $\cN=\cB\setminus \cT$. For $t \in \cT$, let 
$$\cN(t) = \{ n \in \cN \mid n \| t, t \succ n 
\}.$$ 
Let $D = \sqcup_{t \in \cT} \vert \cN(t) \vert$. For each $t \in \cT$ set 
$$ g_t = t - \sum_{n \in \cN(t)} c_{t,n} n$$ 
where $c_{t,n} \in K$. We note that $g_t$ is uniform and that $\tip (g_t) = t$. 
 Set $$\cG_\mathbf{c} = \{ g_t  \mid t \in \cT \}.$$ Note that $\cG_\mathbf{c}$ is not necessarily a \grb basis, since there might be overlaps that do not reduce to zero.

\begin{defn}
With the notation above, let 
$$ \cV_{\cT}  = \{ \mathbf{c} = (c_{t,n}) \in K^D \mid \cG_\mathbf{c} \mbox{ is the reduced \grb basis of } \langle \cG_\mathbf{c} \rangle  \}.$$ 
\end{defn}

If $\mathbf{c}$ is clear from the context, we simply write $\cG$ instead of $\cG_\mathbf{c}$.

\begin{theorem}\label{Correspondence thm}{\rm {\bf [Correspondence Theorem]}}
With the notation above, there is a one-to-one correspondence between the sets $\Alge_{\cT}$ and $\cV_{\cT}$.
\end{theorem}

\begin{proof}
Let $KQ/I \in \Alge_{\cT}$ and $\cG$ be the reduced \grb basis of $I$. Then by hypothesis, $\tip(\cG) = \cT$ and 
there exists a unique array of 
scalars $\mathbf{c} = (c_{t,n}) \in K^D$ such that $\cG =  \{t - \sum_{n \in \cN(t)}  c_{t,n} n \mid t\in\cT\} $.   This gives
 rise to the element $\mathbf{c}  \in \cV_{\cT}$ and $\cG =  \cG_\mathbf{c}$. The inverse map is given by associating $KQ/I_{\cG_\mathbf{c}}$ to the point $\mathbf{c} =(c_{t,n}) \in \cV_{\cT}$ where $\cG_\mathbf{c} = \{ t - \sum_{n \in \cN(t)} c_{t,n} n \mid t \in \cT \}$. 
\end{proof}

We view the one to one correspondence as an identification and freely talk about algebras as elements in $\cV_{\cT}$. 

We now show that $\cV_{\cT}$ is an affine algebraic variety. 

\begin{theorem}\label{thm-var}Let $K$ be a field, $Q$ a quiver, and $\succ$ an admissible
order on $\cB$.  Suppose that $\cT$ is a finite set of paths  and that
$\cN (t)$ is a finite set, for all $t \in \cT$.  Then $\cV_{\cT}$ is an affine algebraic variety.
\end{theorem}

\begin{proof}
Let $x_{t,n}$ be variables, one for each choice of $t\in\cT$ and $n\in\cN(t)$.
We wish to find polynomials in the polynomial ring $K[x_{t,n}\mid t\in\cT, n\in\cN(t)]$ such that
$\cV_{\cT}$ is the zero set of the  polynomials. To simplify notation set $R =   K[x_{t,n}\mid t\in\cT, n\in\cN(t)]$ and view $RQ$ as a subset of $LQ$ where $L$ is the quotient field of $R$.  We apply \grb basis theory to the path algebra $LQ$.

We start by considering the set
\[\cH=\{h_t=t-\sum_{n \in \cN(t)} x_{t,n}n \mid t\in\cT\}
\]
Note that $\cH\subset RQ$.  
For each overlap relation $\mathbf o(h_t,h_{t'},p,q)$, with $h_t, h'_t\in\cH$, we consider its complete
reduction by $\cH$. By Propositon \ref{prop-poly} each such complete reduction has the form
\[ \sum_{n\in\cN} f^*_{\mathbf o(h_t,h_{t'},p,q),n}n,	\]
where the $f^*_{\mathbf o(h_t,h_{t'},p,q),n}$ are polynomials in $R$.

We claim $\cV_{\cT}$  is the zero set  of the set of all the $f^*_{\mathbf o(h_t,h_{t'},p,q),n}$ as one varies
over all overlap relations.  But this is a consequence of Theorem \ref{thm-buch}
since by taking the zero set we obtain exactly the points $(c_{t,n})$ such that
$\{t-\sum_{n\in\cN(t)}c_{t,n}n\}$ is the reduced \grb basis of the ideal it generates.

\end{proof}

\begin{remark}{\rm 
In the proof of Theorem~\ref{thm-var} choosing different sequences of simple reductions for the complete reductions, might give rise to different sets of polynomials.  But the zero sets of these different sets of polynomials are the same  since they correspond to the coefficients in the reduced \grb basis of $I$, for each $I$, and therefore they   are unique.}
 \end{remark} 

Now we restate in a more precise way the list of properties of the algebras in a variety $\cV_{\cT}$. 
\begin{theorem}\label{main result}
Let $Q$ be a quiver  and $\cB$ the set of finite directed paths in $Q$. Let $\succ$ be an admissible order on $\cB$ and let $\cT$ be a tip-reduced  set of paths of length at least 2 such that  $\cN=\cB\setminus \cT$. Let $\Lambda = KQ/I \in \cV_{\cT}$. Set $A_{Mon}=KQ/\langle \cT \rangle$. 
\begin{enumerate}
 \item $A_{Mon} = \Lambda_{Mon}.$
 \item Every algebra has $K$-basis $\cN$, in particular dim$_K$ $A_{Mon} = $ dim$_K$ $ \Lambda$ which is equal to $\vert \cN \vert$. 
 \item $C_\Lambda = C_{A_{Mon}} $  if $I$ is admissible. 
 \item gldim $(\Lambda) \leq $ gldim $(A_{Mon})$ if $I$ is admissible.  
 \item   If gldim $(A_{Mon}) < \infty$  and $I$ admissible then det $C_\Lambda = +1$.
 \item   Assume that  $\Lambda$ is length graded and suppose that $\succ$ respects length. For $N\ge 2$, if  $A_{Mon}$ is $N$-Koszul  then   $\Lambda$ is $N$-Koszul if $I$ is generated in degree $N$.  
 \item Assume that $I$ is admissible. Then  $A_{Mon}$  quasi-hereditary  implies $\Lambda$ quasi-hereditary. Moreover, if, $A_{Mon}vA_{Mon}$ is a heredity ideal then $\Lambda v \Lambda$ is a heredity ideal, where  $v \in Q_0$. 
 \end{enumerate}
\end{theorem}

\begin{proof} (1) follows from the definition of $\cV_{\cT}$ and (2) follows from the Fundamental Lemma, (3) follows from (2), namely that $\cN$ is a $K$-basis for any algebra in $\cV_{\cT}$. 

In order to prove (4), let $\cP$ be the order resolution of a simple $\Lambda$ module as defined in \cite{AG, GS} and as recalled in Section~\ref{sec-order resolution}. Then $\tip(\cP^n) = \tip(\cP^n_{Mon})$, where $\cP_{Mon}$ is the order resolution for the corresponding simple $A_{Mon}$-module. Hence by Corollary~\ref{cor-Betti numbers} the Betti numbers for these resolutions are the same. Noting that by \cite{GHZ} the order resolution for a monomial algebra is minimal, the result follows. 

By \cite{W}, the determinant of $C_{A_{Mon}}$ is equal to one and (5) follows from (4).

For (6), suppose that $I$ is generated in degree $N$ and that  $\Lambda_{Mon}$ is $N$-Koszul.
Then as graded resolutions, the order resolution of $\Lambda$ and $\Lambda_{Mon}$  have projective modules
generated in the same degrees. That is, in both cases,  the $n$-th projective  is generated in degree $(n/2)N$ if $n$ is even and in degree 
  $[(n-1)/2]N+1$ if $n$ is odd. But since the minimal projective resolution of $\Lambda$ is a summand of the order
resolution for $\Lambda$, the result follows. 

Finally  (7) follows from \cite{GSch}.
\end{proof}

We note that a similar result to Theorem~\ref{main result} (6) relating the the Koszul properties of an algebra to its associated monomial algebra has been treated in \cite{CS}.

\begin{Example}{\rm  If none of the elements in $\cT$ overlap then there are no overlap relations and  $KQ/\langle \cT\rangle$ has global dimension 2.  Furthermore, there are no equations to satisfy and it follows
that $\Alge_{\cT}$ is affine space of dimension 
$D = \sum_{t\in\cT} \vert \{ n \in \cN \mid  n\|t \mbox{ and }  t\succ n \} \vert.
$

}
\end{Example}

\begin{Example} \label{local}
{\rm 
Let $K \{ x_1, \ldots, x_n \}$ be the free associative algebra in $n$ variables. We have that $K \{ x_1, \ldots, x_n \} = KQ$, where $Q$ is the quiver with one vertex and $n$ loops.  Let $\succ$ be the length-lexicographical order with $x_n \succ x_{n-1} \succ \cdots \succ x_1$. The commutative polynomial ring in $n$ variables is given by $R_n = K \{ x_1, \ldots, x_n \} / \langle x_i x_j - x_j x_i, \mbox{ for } i > j  \rangle$. The reader may check that
$\{  x_i x_j - x_j x_i \mid i > j  \}$ is the reduced \grb basis for $R_n$. Thus $R_n \in \cV_{\cT}$ where $\cT = \{ x_i x_j \mid i > j\}$.  Note that $\cN = \{ x_1^{a_1} x_2^{a_2} \cdots x_n^{a_n} \mid a_1, \ldots, a_n \in \mathbb{Z}_{\geq 0} \}$. 

We now explicitly compute $\cV_{\cT}$ for $n=2$. In this case $\cT = \{ x_2 x_1 \}$. Since there are no overlap relations to reduce, we have that  $\cV_\cT = K^5$ where $5 = \vert \cN (x_2 x_1) \vert$ 
and $ \cN (x_2 x_1) = \{ x_1x_2, x_1^2, x_1, x_2, 1 \}$. Thus $$\Alg_\cT = \{ K  \{ x_1, x_2 \} / \langle x_2 x_1 - c_1 x_1x_2 - c_2  x_1^2 - c_3  x_1 - c_4  x_2  - c_5  \rangle  \mid (c_1, \ldots, c_5) \in K^5 \}.$$
Note that $R_2$ corresponds to the point $(1,0,0,0,0) \in \cV_\cT$ and that 
all finite dimensional algebras in $\rm{Alg}_\cT$ have global dimension two. 
}
\end{Example}

\begin{Example}{\rm Let $Q $ be the quiver 
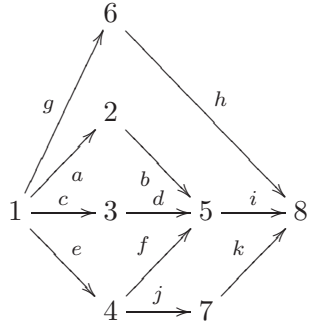
\[
\xymatrix{
&6\ar[ddrr]^h\\
&2\ar[dr]_b\\
1\ar[uur]^g\ar[ur]_a\ar[r]^c\ar[dr]^e&3\ar[r]^d&5\ar[r]^i&8\\
&4\ar[ru]^f\ar[r]^j&7\ar[ru]^k
}\]
Assume $a\succ b\succ c\succ\cdots\succ k$ with $\succ$ being 
the length-left-lexicographic order.
Let $\cT=\{ab,bi,cdi\}$.  Set $g_1=ab-x_1(cd)-x_2(ef),  g_2=bi$ and
$g_3=cdi -x_3(efi)-x_4(ejk)-x_5(gh)$.  Given a point $(c_1,\dots, c_5)\in K^5$, we would like to find  necessary and sufficient conditions
so that evaluating $x_i=c_i$,  $\cG$ is the reduced \grb basis for $\langle \cG \rangle$.  Note that
the theorem says completely reducing all overlap relations obtained from pairs of elements from
$\cG$ results in polynomials in the $x_i$, whose zero
set are the points we want.  The more reductions performed, the higher the degree of 
the resulting polynomials. However,  the degree  of the polynomials is bounded above by the
maximum number plus one of reductions performed to completely reduce the overlap relations.  The
polynomials have constant term 0, so 
$(0,0,\dots,0)$ is in the zero set.  This point gives the monomial algebra $KQ/\langle\cT\rangle$.

In  this example, there is one overlap relation,
$\mathfrak o(g_1,g_2,i,a)= -x_1(cdi)-x_2(efi).$
Now we can reduce $cdi$ to get
\[-x_1(x_3(efi)+x_4(ejk)+x_5(gh))-x_2(efi)=(x_2-x_1x_3)(efi)+(-x_1x_4)(ejk)-x_5(gh).\]

Thus, the variety in $K^5$ of the ideal  $\cA=  \langle x_2-x_1x_3, x_1x_4, x_5 \rangle$   is in one-to-one correspondence
with $\Alge_{\cT}$  where
\begin{align*} \Alge_{\cT} = \{ KQ / \langle ab-c_1(cd)-c_2(ef), & bi, cdi -c_3(efi)-c_4(ejk) - c_5(gh) \rangle  \mid  & \\  & (c_1,\dots, c_5)\in K^5  
\mbox{ and } c_2 - c_1c_3 =0, c_1c_4 =0, c_5 = 0 \}. \end{align*}
}
\end{Example}

We end this sections with a result giving sufficient conditions for a variety of algebras to contain a product of varieties of algebras.

\begin{proposition}  Let $\cT $ a tip-reduced subset of paths in a quiver $Q$. 
Suppose that $\cT$ is a disjoint union of two sets $\cT_1$ and $\cT_2$ with no overlaps between the elements in $\cT_1$ and $\cT_2$. Then the variety $\cV_\cT$ contains the variety $\cV_{\cT_1} \times \cV_{\cT_2}$.
\end{proposition}

\begin{proof}
Note that if $t \in \cT$ then $t \in \cT_1$ or $t \in \cT_2$. Hence if $n$ is parallel either to a paths $t_1 \in \cT_1$ and smaller than $t_1$ or it is parallel to a path $t_2 \in \cT_2$ and smaller than $t_2$ but not both. 
Now suppose that we have a point $\mathbf{c_1} \in \cV_{\cT_1}$ and $\mathbf{c_2} \in \cV_{\cT_2}$ and let $\mathbf{c}$ be the point  defined by 
$$ \mathbf(c) = (c_{t,n}) \mbox{ where } c_{t,n} = \left\{ \begin{array}{lll} c_{t_1, n} & \mbox{ if } &t=t_1 \\
c_{t_2, n} &\mbox{ if } &t=t_2. \end{array} \right. $$ 
It follows that $\mathbf{c}$ is in $\cV_\cT$ since the set $\{ g_{t_i} = t_i- \sum_{n || t_i, t_i \succ n } c_{t_i,n} n \mid t_i \in \cT_i\}$  is a \grb basis of the algebra associated to the point  $\mathbf{c_i}$, for $i = 1,2$. 
\end{proof}

\section{Applications to algebras of global dimension two}\label{sec:gldim2}

In this Section we study varieties of algebras of global dimension two  and we show that for algebras of global dimension two the variety is an affine space. We also show that for every variety $\cV_{\cT}$, where $\cT$ is a set of non-overlapping paths,  and for any ordering of the elements of $\cT$, we obtain a flag of vector spaces each of which corresponds to a variety $\cV_{\cT'}$ where $\cT' \subset \cT$. 
Restricting to algebras of global dimension two over the complex numbers, we show that any algebra of global dimension two continuously deforms to its associated monomial algebra.

\begin{proposition}\cite{GHZ}\label{prop-noverap}\label{prop-nooverlap}Let $\cT$ be a tip-reduced nonempty set of paths in $\cB$
such that $K\cQ/\langle  \cT\rangle$ is a finite dimensional    $K$-agebra.
Then $\gldim K\cQ/\langle  \cT\rangle=2$ if and only if,  for
all paths  $p$ and $q$ in $\cT$, $p$ does not overlap  $q$.

\end{proposition}

If, for all $p,q\in\cT$, $p$ does not overlap $q$, we say \emph{$\cT$ contains
no overlaps}.

\begin{proposition}\label{prop-ambient} Let $\cT$ be a set of paths of length at least 2 with
no overlaps; that is,
$\gldim K\cQ/\langle \cT\rangle=2$.  Then $ \cV_{\cT}=K^{D_{\cT}}$.
In particular, $\cV_{\cT}$ is the ambient affine space.

\end{proposition} 

\begin{proof}In the construction of $\cV_{\cT}$, we see that the variables
	$c_{n,t}$ have no restrictions because there are no overlaps
of elements of $\cT$. 
\end{proof}

 In summary for algebras of global dimension two we have the following.

\begin{theorem}\label{gldim2} Let $\cT$ be a finite 
nonempty tip-reduced subset of  $\cB$
such that $K\cQ/\langle \cT \rangle$ is a  finite dimensional $K$-algebra.  The following
statements are equivalent:
\begin{enumerate}
\item $\gldim K\cQ/\langle \cT \rangle=2$.

\item  Every algebra corresponding to a point in $\cV_{\cT}$
has global dimension 2.

\item There are no	overlaps amongst elements of $\cT$.

																\end{enumerate}
If (1)-(3) hold, then the affine variety $\cV_{\cT}$ coincides with the ambient space $K^{D_{\cT}}$.
\end{theorem}

\begin{proof}

All implications apart from (1) implies (2) follow from Propositions~\ref{prop-nooverlap} and~\ref{prop-ambient}.  We now show that (1) implies (2). By assumption
 $\gldim K\cQ/\langle  \cT\rangle=2$. 
Suppose that $A$ is an algebra associated to a point in $\cV_{\cT}$.
By Theorem~\ref{main result} we have $\gldim A\le 2$. By \cite{GHZ}, since $\cT$ is nonempty,
the global dimension of $A$ cannot be 0 or 1. 
\end{proof}

Recall from \cite{GH}  that a quadratic algebra that admits a quadratic \grb basis is a Koszul algebra. In \cite{G Koszul}, Koszul algebras with quadratic \grb basis have been called \emph{ strong Koszul algebras}. 

\begin{corollary}
Let $\cT$ be a set of non-overlapping paths of length two. Then every algebra corresponding to a point in $\cV_\cT$ is a strong Koszul algebra. 
\end{corollary}

\begin{corollary}
Let $\cT$ be a set of non-overlapping paths of length two and suppose that $K$ is algebraically closed. Then any algebra in $\cV_\cT$ continuously deforms to its associated monomial algebra. 
\end{corollary}

\begin{proof}
Since by Theorem~\ref{gldim2} the variety $\cV_\cT$ is the ambient affine space, any algebra in the associated variety  continuously deforms to its associated monomial algebra.  
\end{proof}

 We now construct flags associated to a variety $\cV_\cT$  where the elements in $\cT$ do not have any overlaps.  Each flag depends on a total order on  $\cT$.  More precisely, let $\cB$ be the basis of paths of $KQ$ and let $\succ$ be an admissible order on $\cB$. Now fix a finite set $\cT \subset \cB$  such that the elements in $\cT$ do not have any overlaps and order the elements of $\cT$ as $t_1 > \ldots > t_n$. We remark that a  canonical ordering is the order on $\cT$ induced by $\succ$, however, for the moment we consider any order $>$ on $\cT$. Then if the variety $\cV_{\cT} = {0}$ then the flag is trivial. So suppose that $i$ is smallest integer such that $\cV_{\{t_1, \ldots, t_i\}} \simeq K^{n_1}$ where $n_1$ is the number of paths in $\cN$, parallel to $t_i$ and smaller than $t_i$ in the order $>$ and set $\cT_1 = \{t_1, \ldots, t_i\}$. That is, the elements in  $\{ t_1, \ldots, t_{i-1}\}$ have  no parallel paths in $\cN$  which are smaller in the order $\succ$. If $\cT_1 = \cT$ then the flag corresponding to the inclusion  $0 \subset \cV_{\cT}$ is given by $0 \subset K^{n_1}$. Otherwise define $\cT_2 = \{t_1, \ldots, t_j \}$ to be the smallest $j$ bigger than $i$ such that $\cV_{\cT_2}\simeq K^{n_2}$ with $n_2 > n_1$.  Continuing in this fashion we get a sequence of varieties $\cV_{\cT_1},  \cV_{\cT_2}, \cdots, \cV_{\cT_k} = \cV_{\cT}$ corresponding to the flag 
$$ 0 \subset K^{n_1} \subset K^{n_2} \subset \cdots \subset K^{n_k}.$$

We call $\cT$ \emph{minimal} if for every element of $\cT$ there exists a parallel element $n \in \cN$ such that $t \succ n$.   As a consequence we obtain the following result. 

\begin{proposition}
Let $\cT$ be a tip-reduced subset of $\cB$ such that there are no overlaps in $\cT$. Then there exists a unique minimal tip-reduced subset $\cT'$ of $\cT$. Furthermore, $\cV_{\cT'} \simeq \cV_{\cT}$ and thus for a given order on $\cT$ and the induced order on $\cT'$ the corresponding flags are isomorphic. 
\end{proposition}

We end this section with an example of flags associated to an algebra of global dimension two
and demonstrate for such an algebra how to construct a flag, given an order on a given tip-reduced subset $\cT$ of $\cB$.  It is also  an example which shows, in general, that one obtains different flags by varying the order on $\cT$.

\begin{Example}{\rm 
Let $Q$ be an acyclic quiver with one source and one sink and 10 non-trivial disjoint paths from the source to the sink, that is all paths  have the same start and end point but no other vertex in common. Denote  six of these paths by $p_1, \ldots, p_6$ and the remaining four by $q_1, \ldots, q_4$. 

Let $\succ$ be an admissible order on the paths of $Q$
with $p_6 \succ p_{5}  \succ q_4 \succ q_3  \succ p_4 \succ q_2 \succ p_3 \succ p_2 \succ q_1 \succ p_1$.
Set
 $\cT=\{p_1,p_2, \dots,
p_6\}$. Then  $\cN_\cT = \{q_1, \ldots, q_4\}$.

We consider two orderings of $\cT$ and the induced flags. First define
 $p_5 > p_6>p_3 > p_4 > p_1 > p_2$. Then the corresponding vector spaces
in the above construction are given by  $V_1 =  \cV_{p_1, p_2} = K$, $V_2 = \cV_{p_1, p_2, p_4} = K^3   $, $V_3 = \cV_{p_1, p_2, p_4, p_3}  = K^4$, $V_3 = \cV_{p_1, p_2, p_4, p_3, p_6}  = K^8$, $V_4 = \cV_{p_1, p_2, p_4, p_3, p_6, p_5}  = K^{12}$  
and the flag induced by the order $>$ is given by 
$$ 0 \subset K \subset K^3 \subset  K^4 \subset K^8 \subset K^{12} $$

Next take the ordering of $\cT$ induced by $\succ$; namely,
$p_6\succ p_5\succ p_4\succ p_3 \succ p_2\succ p_1$.
The vector spaces and flag  induced by the admissible order $\succ $ are 
$V_1 = \cV_{p_1} = 0$, $V_2 = \cV_{p_1, p_2} = K$, $V_3 = \cV_{p_1, p_2, p_3} = K^2$, $V_4 = \cV_{p_1, p_2, p_3, p_4} = K^4$, $V_5 = \cV_{p_1, p_2, p_3, p_4, p_5} = K^8$, $V_6 = \cV_{p_1, p_2, p_3, p_4, p_5, p_6} = K^{12}$ and 
$$ 0 \subset K \subset  K^2 \subset  K^4 \subset K^8 \subset K^{12}. $$

}
\end{Example}

\begin{remark}
{\rm

	The above example was chosen to illustrate the construction of flags and to show how they depend on the chosen order on $\cT$.  It is easy to see how to generalise the example to an arbitrary number of $p_i$ and $q_j$. These quivers contain many well-known classes of algebras, such as Kronecker and canonical algebras. }\end{remark}}


\section{Special subvarieties,  varieties of graded algebras and varieties of algebras defined by admissible ideals}


In this section we define and investigate a class of subvarieties of the varieties defined in the previous section. We call these  \emph{special subvarieties}. We will then  see
 that varieties of graded algebras can be seen as special subvarieties.

We start by defining the class of special subvarieties. 
For each $t\in\cT$, let $\cS(t)$ be a subset of $\cN(t)$, where $\cS(t)$ may be empty.
We assign values to the elements of $\cS(t)$ via a set map $\varphi_t\colon \cS(t)\to K$  and let $\varphi = \{ \varphi_t \mid t \in \cT \}$.

\begin{definition} With the notation above, let
$$\cV_\cT(\varphi) = \{ ( {\mathbf c})=(c_{t,n}) \in \cV_\cT \mid  \text{ for all } t \in \cT, 
\text{ if } n\in \cS(t) \text{ then }c_{t,n}=\varphi_t(n) \}$$
\end{definition}

Thus, $\cV_{\cT}(\varphi)$ are the points in $\cV_{\cT}$ where, for each $t\in\cT$ and
$n\in \cS(t)$, $c_{t,n}$ is restricted to the value $\varphi_t(n)$.

\begin{proposition}
 For $\varphi = \{ \varphi_t \mid t \in \cT \}$ as above,  the set  $\cV_\cT (\varphi)$ is a subvariety of $\cV_\cT$. 
\end{proposition}

\begin{proof}
Let $(f_1, \ldots, f_m)$ be the ideal in $K[x_{t,n} \mid t \in \cT, n \in \cN]$ of the variety
 $\cV_\cT$. Then $\cV_\cT(\varphi)$ is the zero set of the polynomials $\{f_1, \ldots, f_m\} \cup (\cup_{t\in \cT} \{ x_{n,s} -\varphi_t(n) \mid n\in \cS(t) \})$.  The result follows.
\end{proof}

We call $\cV_\cT(\varphi)$ \emph{the special subvariety of $\cV_{\cT}$ associated to $\varphi$}. 

\begin{corollary} $\cV_\cT(\varphi)$ contains the monomial algebra $KQ / \langle \cT \rangle$ if and only if  either $ \varphi_t \equiv 0$  or $\cS (t) $ is empty, for each $t\in\cT$.
\end{corollary}


\subsection{Varieties of graded algebras}


In this subsection we show that the connection between algebras and varieties extends to the case of graded algebras.  We consider algebras graded by weight functions. 

Recall that a \emph{weight function } on a path algebra $KQ$ is a function $W : Q_1 \to G$ where $G$ is a group. 
We extend $W$ from arrows to $\cB$ as follows: if $p = a_1 \ldots a_n$ then $W(p) = W(a_1) \ldots W(a_n)$ where $W(v) = 1_G$ for $v \in Q_0$. Then $W$ induces a $G$-grading on $KQ$ where 
$$KQ_g = {\rm Span_K}  \{ b \in \cB \mid  W(b) = g \}.$$  

There is an induced $G$-grading on $\Lambda=KQ/I$ if $I$ can be generated by weight homogeneous elements. We call this grading the \emph{induced grading on $\Lambda$}. 
If $I$ is generated by $W$-homogeneous elements then the reduced \grb basis consists of $W$-homogenous elements. 

Note that, for any weight function, monomial algebras  always  have an induced grading.  Moreover, our theory includes $\mathbb Z$-graded  algebras  where the grading is induced by  a $\mathbb Z$-grading on the arrows. 

\begin{defn} Let $\cT \subset \cB$ be a tip-reduced set of elements of length at least two and let $\succ$ be an order on $\cB $. Let $W: Q_1 \to G$ be a weight function. 
 Define
$$\Alge_{\cT}^W = \{ KQ/ I \mid   I \mbox{ weight homogeneous ideal of }  KQ \mbox{ and } 
\Imo=\langle \cT \rangle \}.$$
\end{defn}

 Define 
\[\cN^{W}(t)=\{n\in\cN\mid n\|t, \ell(n)\ge1, t\succ n, W(n)=W(t)\}.\]
 where $\ell(n) = k$ if $n = a_1 \ldots a_k$ with $a_i \in Q_1$. 

Analogously to the ungraded case, set $\cA=K^{D^W}$ where $D^W=\sum_{t\in\cT}| \cN^{W}(t) |$.  For each $t \in \cT$ set 
$$ g_t = t - \sum_{n \in \cN^W(t)} c_{t,n} n$$ 
where $c_{t,n} \in K$. We note that $g_t$ is uniform, weight homogeneous and that $\tip (g_t) = t$. 
Let $\cG = \{ g_t \mid t \in \cT \}$. Note that,
 in general $\cG$ is not necessarily a \grb basis for the
ideal it generates, since there might be overlaps that do not reduce to zero.

\begin{defn}
With the notation above, let 
$$ \cV_{\cT}^W  = \{ \mathbf{c} = (c_{t,n}) \in K^{D^W} \mid \cG \mbox{ is the reduced \grb basis of  } \langle \cG \rangle \}$$ 
\end{defn}

\begin{theorem}\label{Correspondence thm}{\rm {\bf [Graded Correspondence Theorem]}}
With the notation above, there is a one to one correspondence between the sets $\Alge_{\cT}^W$ and $\cV_{\cT}^W$.
\end{theorem}

The proof of the theorem  is  analogous to the proof in the non-graded case. The next result shows that 
$\cV_{\cT}^W$ is  an affine algebraic variety and in fact that it is a special subvariety of $\cV_{\cT}$.

\begin{theorem}\label{thm-var graded}Let $K$ be a field, $Q$ a quiver, and $\succ$ an admissible
order on $\cB$.  Let $W: Q_1 \to G$ be a weight function. Set $\cT$ to be  a finite set of paths of lengths at least 2. 
 For each $t\in\cT$, let $\cS(t) = \cN(t) \setminus  \cN^W(t)$ and $\varphi_t\colon
 \cS(t) \to K$ such that $\varphi_t \equiv 0$.  Then,  for  $\varphi = \{ \varphi_t \mid t \in \cT \}$, 
 $$\cV^W_{\cT} \simeq \cV_\cT (\varphi)$$ 
 and  thus $\cV^W_{\cT}$ is  a special subvariety
of $\cV_{\cT}$ and therefore $\cV^W_{\cT}$ is an affine algebraic   variety.  
\end{theorem}

\begin{proof}  
For each $t\in\cT$, let  $\cS(t)=  \cN(t) \setminus  \cN^W(t)= \{n\in \cN(t)\mid W(n)\ne W(t)\}$ and let 
$\varphi_t\colon \cS (t) \to K$ given by  $\varphi_t\equiv 0$.  
  Define $\Psi\colon\cV_{\cT}^W\to \cV_{\cT}(\varphi)$ by sending
$\mathbf c^W=(c^W_{t,n})$  to $\mathbf c =(c_{t,n})$
where $c_{t,n}=c^W_{t,n}$ for $n\in\cN(T)^W$ and $0$ otherwise.

It is clear that $\Psi$ is one-to-one.   To see that
$\Psi$ is onto, let  $\bf c=(c_{t,n})\in \cV_{\cT}(\varphi)$.   For $t\in \cT$  and $n\notin \cS(t)$ let
$c^W_{t,n}=c_{t,n}$, noting that by construction,
$W(t)=W(n)$.  It is immediate that
$\mathbf c^W = (c^W_{t,n})\in \cV_{\cT}^W$ and $\Psi((c^W_{t,n}))=
\bf c$, proving that $\Psi$ is onto; hence a  bijection.  
\end{proof}	 

Note that  since in Theorem~\ref{thm-var graded}  $\varphi_t \equiv 0$ for all $ \varphi_t \in  \varphi $, the monomial algebra $KQ / \langle \cT \rangle$ is contained in $\cV^W_{\cT}$. The next result follows from Theorems~\ref{main result} and~\ref{thm-var graded}.

\begin{corollary}\label{main result graded}

Let $\Lambda = KQ/I \in  \Alge_{\cT}^W$,  that is,  in particular, $I$ is a weight homogeneous ideal. 
Then
 $$KQ/\langle \cT \rangle \simeq \Lambda_{Mon}  \mbox{ as $G$-graded algebras.}$$ In particular, $(KQ/\langle \cT \rangle)_g \simeq   \Lambda_g$  as $K$-vector spaces, for all $g \in G$, and  dim$_K  \Lambda_g = \vert \{ n \in \cN  \mid W(n) = g\} \vert$.
 \end{corollary}

Recall that if $\Lambda$ is a positively $\mathbb{Z}$-graded algebra, the Hilbert series of $\Lambda$ is $$ {\Hilb}(\Lambda) = \sum_{ n=0}^\infty \dim \Lambda_n t^n.$$  The next result follows directly  from Corollary~\ref{main result graded} 

\begin{corollary}
Let $\Lambda, \Lambda'  \in \cV^W_{\cT}$ where $\Lambda, \Lambda'$ are positively $\mathbb Z$-graded algebras, that is $G = \mathbb Z$ and $W\equiv 1$. Then 
$\Hilb(\Lambda) = \Hilb(\Lambda')$. 
\end{corollary}

\begin{proof} The result follows from the fact that they  have the same bases  given by  $\cN=\cB\setminus \tip(\langle \cT\rangle)$ and hence  $\dim_K(\Lambda_s)=\dim_K(\Lambda'_s)=\vert \{n\in\cN\mid W(n)=s\} \vert$.

\end{proof}

\begin{Example}\label{Example AS graded} \rm{ We continue Example~\ref{local}.  Recall that the associated monomial algebra of the commutative polynomial ring $R_n = K\{ x_1, \ldots, x_n \} / \langle x_i x_j - x_j x_i \mid i > j \rangle$  is $K\{ x_1, \ldots, x_n \} / \langle x_i x_j \mid i > j \rangle$ and $\cN = \{x_1^{i_1} \cdots x_n^{i_n} \mid i_1, \ldots, i_n \geq 0 \}$. 
Let $W: Q_1 \to \mathbb{Z}$ defined by  $W (x_1) = \cdots = W(x_n ) = W(1) = 1$. 
Therefore every algebra in $\cV_\cT^W$ has Poincar\'e-Birkhoff-Witt basis $\cN$, it is Koszul, has global dimension $n$, has as $k$-th Betti number the binomial coefficient $C(n,k)$ and has the same Hilbert series as the commutative polynomial ring.

 Again,  as in Example~\ref{local}, we consider the case $n=2$ in more detail. Let  $\cT = \{ x_2 x_1 \}$ and   $W: Q_1 \to \mathbb{Z}$ be as above. Then
 $$ \cN^W (x_2x_1) = \{ x_1x_2, x_1^2 \}$$
 and  $$\cV^W_\cT = \{ (c_1, c_2) \mid c_1, c_2 \in K \}$$
 where the point $(c_1, c_2) \in K^2$ 
 corresponds to $K\{ x_1, x_2 \} / \langle x_2x_1 - c_1 x_1x_2 - c_2 x_1^2 \rangle$ in  $\rm{Alg}^W_\cT$. These are all Koszul algebras with quadratic  \grb basis (see \cite{G Koszul}), with associated monomial algebra $K \{ x_1, x_2 \} / \langle x_2x_1 \rangle$. It is interesting to note that $K \{ x_1, x_2 \} / \langle x_2x_1 \rangle$ is the only graded algebra of global dimension 2 and Gelfand-Kirillov-dimension 2 which is not Noetherian \cite{AS}.

 Let  $\cS (t) = \cN (t) \setminus \cN^W (t)$ for all $t \in \cT$ and let $\varphi_t : \cS (t) \to K$ be given by $\varphi_t \equiv 0$. Then for $\varphi = \{\varphi_t  \mid t \in \cT \}$, we have 
$$
 \cV^W_\cT \simeq  \cV_\cT (\varphi)  = \{ (c_1, c_2, 0, 0, 0) \mid c_1, c_2 \in K \}.$$

The Artin Schelter regular  algebras of global dimension two fall into two families.
 Namely,  the first family consists of the commutative polynomial ring and the quantum affine planes
 $$ A_{(c_1,0)} : = K\{ x_1, x_2 \} / \langle x_2 x_1 - c_1 x_1 x_2 \rangle, \mbox{ with }  \neq  c_1 \in K$$  and secondly the algebra $$ A_{(1,1)} = K\{ x_1, x_2 \} / \langle x_2 x_1 -  x_1 x_2 - x_1^2 \rangle.$$
 It follows from the above  that  $A_{(c_1,0)}, c_1 \in K$  and $A_{(1,1)}$ are  in $\cV^W_\cT$. As elements corresponding to points in $\cV_\cT(\varphi)$, we write $A_{(c_1, 0, 0, 0, 0)} $ for $K\{ x_1, x_2 \} / \langle x_2 x_1 - c_1 x_1 x_2 \rangle$ and $A_{(1, 1, 0, 0, 0)} $ for $K\{ x_1, x_2 \} / \langle x_2 x_1 -  x_1 x_2 - x_1^2 \rangle.$
 
We   can separate these two families of Artin Schelter regular algebras  }  into disjoint subvarieties  of $\cV_\cT$.  For example,  if we let $\cS^1(x_2 x_1 ) =  \{ x_1^2, x_2, x_1, 1\}$ and $$\varphi_1 = \varphi^1_{x_2x_1}: \{ x_1^2, x_2, x_1, 1\} \to K   $$ such that $\varphi^1_{x_2x_1}\equiv 0$, then $A_{(c_1,0,0,0,0)}  \in \cV_\cT (\varphi_1)$ but $A_{(1,1,0,0,0)} \notin \cV_\cT (\varphi_1)$. Now define $$\varphi_2 = \varphi^2_{x_2x_1}: \{  x_1^2, x_2, x_1, 1\} \to K   $$  by $\varphi_2(x_1^2 ) = 1$ and $\varphi_2(x_1) = \varphi_2(x_2) = \varphi_2(1) =0$. Then $A_{(1,1,0,0,0)}  \in \cV_\cT (\varphi_2)$. Furthermore, $
  \cV_\cT (\varphi_1) \cap \cV_\cT (\varphi_2)
 = \emptyset$ and the two subvarieties are lines in $K^5$.   
 
\end{Example}


\subsection{Admissible ideals}\

Admissible ideals play an important role in the representation theory of
finite dimensional algebras. 
Recall that an ideal $I$ in $KQ$ is
admissible if $J^m\subseteq I\subseteq J^2$  for some $m\ge2$,
where $J$ is the ideal generated by the arrows.  The varieties and subvarieties we have
constructed up to now, do not allow
one to only study
admissible ideals.  In this subsection,   we show how to remedy this by constructing
 varieties where all the corresponding algebras  are defined by admissible ideals.  More precisely, given $\cT$ a set   of tip-reduced
paths in a quiver and $m \geq 2$,  we would like to have a variety whose points
correspond to algebras $KQ/I$ with $I$ admissible.   We
show   how to construct  such an affine algebraic variety, which we call $\cV^{ad}_{\cT} (m)$ such that $\cV^{ad}_{\cT}(m)$
 is a subvariety of $\cV_{\cT}$ with
the points in one-to-one correspondence with the set of algebras $A=KQ/I$, $A\in\cV_{\cT}$ such that $J^m\subseteq I\subseteq J^2$.

  We start by giving  an example of an inadmissible ideal such that the associated monomial
ideal is  admissible. Take $I=\langle x^3-x^2\rangle$ in $K[x]$.  Clearly,  $K[x] /\langle x^3-x^2\rangle $ is finite dimensional but $\langle x^3-x^2\rangle$ is not admissible since there is no $n \geq 2$ such that $x^n \in \langle x^3-x^2\rangle$.  We  see, however,  that $x^3-x^2$ is the reduced \grb
basis for $I$ under any admissible order and that
$\tip(\langle x^3-x^2\rangle) =\{ x^n\mid n\ge 3\}$. Then $\cN=\{x^n\mid n\le 2\}$
and the associated monomial ideal is $\langle x^3\rangle$, which
is admissible.  Thus simply requiring that $KQ/I$ be finite dimensional
and $I\subseteq J^2 $ is not sufficient for $I$ to be admissible.

Suppose we are given $\cT$, a tip-reduced set of  paths in $J^2$.
Let $\cN=\cB\setminus\tip(\langle \cT\rangle)$ and 
assume that  $\cN\subseteq \cB_{\le m-1}$, where $\cB_{\le m-1}$ denotes the set of  paths of length $\le m-1$ ; in particular, $\cN$ is finite.  Note that  this assumption is necessary since  if $\cN\not
\subseteq \cB_{m-1}$ for all $m$,  then, for all $m$,  $J^m\not\subset \langle \cT\rangle$. For $t\in\cT$, set
$\cS(t)=\{n\in \cN(t)\mid \ell(n)\le 1\}$, and $\varphi_t\colon \cS(t)\to K$
with $\varphi_t \equiv 0$.   Let $\varphi = \{ \varphi_t \mid t \in \cT \}$.  
Then $\cV_{\cT}(\varphi)$ is the variety whose points correspond
to the algebras $KQ/I$, $I\subseteq J^2$
with reduced \grb basis $\cG$ with $\tip(\cG)=\cT$.
Let $A=KQ/I\in \cV_{\cT}(\varphi)$ with reduced \grb basis $\cG$.
Then $J^m\subseteq I$ if and only if each path of length $m$ is in $I$. 
That is, by Proposition~\ref{prop-red2zero}, $J^m\subseteq I$
if and only if each path of length $m$ reduces to 0 by $\cG$.

{\bf Definition of the subvarieties $\cV_{\cT}^{ad}(m)$ containing algebras defined by admissible ideals:} Given $m \geq 2$ and $\varphi$ as above,   we construct a variety  $\cV_{\cT}^{ad}(m)$ as a subvariety of the special subvariety  $\cV_{\cT}(\varphi)$  of $\cV_\cT$ such that the points of $\cV_{\cT}^{ad}(m)$ correspond to algebras of the form $KQ/I$ such that $J^m \subseteq I \subseteq J^2$.    Since we are considering admissible ideals, we can assume  that $\cN$ is finite and that $m$ is such that 
  $\cN\subseteq \cB_{\le m-a}$ for some positive integer $a$.

Recall that given
$\cN (t)$, we consider $({\bf c})=(c_{t,n})$ in affine $D$-space 
where $D=\sum_{t\in\cT}|\cN (t)|$. Since we are in
$\cV_{\cT} (\varphi)$, the coefficient of an $n\in\cN (t)$ of length $<2$
 is 0. Consider \[\cH=\{t-\sum_{n\in\cN (t) \text{ and }\ell(n)\ge 2}
x_{n,t}n\},\]
where the $x_{n,t}$ are variables.

As before, we completely reduce each overlap relation
by $\cH$ and let $f_1,\dots f_r$ be the polynomial coefficients of the paths in $\cN$.  The zero set of the ideal generated by the $f_i'$s is
$\cV_{\cT}(\varphi)$.   Next, completely reduce each path of length $m$ by
$\cH$.     Let $f_1^*,\dots, f^*_s$  be the polynomial coefficients of the paths in $\cN$ obtained from these reductions.  The zero set
of the ideal generated by the $f_i$'s and the $f^*_j$'s is the desired
variety $\cV_{\cT}^{ad}(m)$.  
                                                                                                                
It is clear that if $\Lambda=KQ/I$ and $I$ such that $J^m \subseteq I \subseteq J^2$ with
reduced  \grb basis $\cG$ and $\tip(\cG)=\cT$, then $\Lambda$
corresponds to a  point in $\cV_{\cT}^{ad}(m)$.

Summarizing  the above, we have proved the following. 

\begin{proposition} Let $m \geq 2$ and $\cT$ be a tip-reduced set of   paths in
$Q$ of length at least 2 and $J^m\subseteq\langle \cT\rangle$.  Then $\cV_{\cT}^{ad}(m)$ is a subvariety of
$\cV_{\cT}$ and thus an affine algebraic variety. The points of  $\cV_{\cT}^{ad}(m)$ are in one-to-one
correspondence with finite dimensional algebras of the form $KQ/I$
where
\begin{enumerate}
\item  $I$ is an admissible ideal with $J^m \subseteq I \subseteq J^2$,
\item  the reduced \grb basis  has tip set $\cT$.
\end{enumerate}

\end{proposition}

 Coming back to the example above, recall that we had $\cT=\{x^3\}$
and that $\langle x^3-x^2\rangle$ is inadmissible.  We show how
the above construction handles this case,  that is we exam the special subvarieties with associated monomial ideal $\langle x^3 \rangle$. 
For $t = x^3$, we have that  $\cN (t) = \{ x_2, x, 1 \}$, $\cS(t) = \{ x, 1 \}$  and  $\varphi_{t} \equiv 0$.  
  Thus $\cH=\{x^3-Ax^2\}$ where $A$ is a variable.
 If we set  $m = 4$ then we need to reduce $x^4$ by $\cH$.
viewing $x^4$ as $x^3x^1$ we see  that
$x^4$ reduces to $Ax^2\cdot x=Ax^3$.
Now, $Ax^3$ reduces to $A^2x^2$.  Thus
$A^2x^2$ is a complete reduction of $x^4$
by $\cH$.  Hence the only new requirement
is the zero set of $A^2$.  Thus $A$ is 0 as
expected.  Hence the only admissible ideal
so that  $I_{Mon}=\langle x^3\rangle$ is
$\langle x^3\rangle$.

We will end this section with an example of a non-trivial special subvariety generated by admissible ideals. 

Let $Q$ be the quiver
\xymatrix{
v_1\ar@(ul,dl)[]^a\ar@/_1pc/[r]_d\ar[r]^c\ar@/^1pc/[r]^b& v_2
}
and $\cT=\{a^3,a^2b\}$.  Then we see that\linebreak
 $\cN=\{v_1, v_2, a,b,c,d,a^2,
ab,ac,ad,a^2c,a^2d\}$. 
Let $\succ$ be the length-lexicographic order with
$a\succ b\succ c\succ d$. Then $\cV_{\cT}$ consists of algebras
having $K$-dimension 12 since $\vert\cN\vert=12$  and $\cN\subseteq \cB_{\le 3}$.

Since we are interested in admissible ideals, if $J$ denotes the ideal in $KQ$ generated by  the arrows, then we specialize
to the case where $I\subseteq J^2$.  Thus the coefficients
of the vertices and arrows are all fixed to be 0 in the elements  of the reduced \grb basis. So in fact we are considering the special subvariety $\cV_\cT(\varphi)$ with $\varphi = \{ \varphi_t \mid t \in \cT \}$ where $\varphi_t: \cS(t) \to 0$ and $\cS(t) = \{ n \in \cN(t) \mid n || t, \ell(n) \leq 1 \}$. 

Set $g_1=a^3-Sa^2$ and $g_2=a^2b-Ta^2c-Ua^2d-Vab -Wac -Xad$, where $S,T,U,V,W,X$ are variables.  We see that
 $\cV_{\cT}(\varphi)$    is a variety in $K^6$.

There are four overlap relations; namely $a^3$ overlaps
with itself in $a^2$ and also  in $a$, and $a^3$ overlaps
with $a^2b$, again in $a^2$ and in $a$.  Completely reducing
these overlap relations, we obtain the following  polynomials
as coefficients:

\[ VT+W, VU+X, -SV+V^2, -SW+WV, -SX + XV, S(VT+W), \] \[ S(V^2 - SV), SW (V-S), SX(V-S)\]

Thus, $\cV_{\cT}(\varphi)$  is the zero locus of the above polynomials.

We now find $\cV_{\cT} ^{ad}(m)$ for $m=4$. There
are 4 paths of length 4:  $a^4, a^3b, a^3c, a^3d$.
Since we want $J^4 \subseteq I$, each of the four paths needs  to reduce to 0 by $\{g_1,g_2\}$ and therefore we obtain additional polynomial equations in the coefficients. 
For example, $a^4$ reduces to $S^2a^2$ and we conclude that
$S=0$.  But then it already follows that all the paths of length 4 reduce to 0
since $g_1=a^3$.
So finally, if $S=0$ then $V=0$ from the equation
$-SV+V^2=0$.    Similarly we find that $W=X=0$ and that $\cV_{\cT}^{ad}(4)$ is affine 2-space with $T,U$ chosen freely.

 We note that if $\Lambda$ is length graded and if $\succ$ respects the length grading,  then 
since the monomial algebra $KQ/\langle \cT \rangle$ is 3-Koszul, by Theorem \ref{main result} (6)
every algebra in $\cV^W_{\cT}$ is $3$-Koszul.


\section{Appendix: Order resolutions}\label{sec-order resolution}


One of the main tools for proving the statements in Theorem 1.1 is based on the algorithmically constructed projective resolutions in \cite{GS}, which we recall in this appendix under the name of \emph{order resolutions}. 
In \cite{AG} and \cite{GS} two methods for creating projective resolutions of modules were
presented.  Although seemingly different, they turn out to be the same resolution.  We
adopt the approach found in \cite{GS} which employs  
\grb bases.  We call these resolutions \emph{order resolutions} since
they are dependent on the chosen admissible order on the paths in $Q$.  We note that if $\Lambda$ is not monomial then the order resolution of a finitely generated $\Lambda$-module is not necessarily minimal. We will see however, that for a monomial algebra, the order resolution is the minimal projective resolution. 

We fix a quiver $Q$, an admissible order $\succ$ on $\cB$, a tip-reduced
subset $\cT$ of $\cB$, and a vertex $v\in Q_0$.
Let $\Lambda$ and $\Lambda'$ be two algebras in $\cV_{\cT}$.  The goal of this
section is to compare the order resolutions of the simple $\Lambda$-module $S_v$
to the order resolution of the simple $\Lambda'$-module $S'_v$, where $S_v$ and $S'_v$
are the one dimensional simple modules at $v$.  For this we compare the order resolutions
where $\Lambda'=\Lambda'_{Mon}=\Lambda_{Mon}$ with the last equality due to
the assumption that both
algebras are in $\cV_{\cT}$.

We begin by
recalling the general framework of  the algorithmic construction of order resolutions of modules over $\Lambda=KQ/I$.  Our specific goal is to show that the Betti numbers of the order resolutions of $S_v$ and $S_v'$ are the same.

Let $P$ be a finitely generated projective $KQ$-module. We fix a
direct sum decomposition of   $P = \bigoplus_{i \in \cI}v_i KQ$, where $v_i \in Q_0$ not necessarily distinct and $\cI$ is a  finite indexing set \cite{G}. We write elements of $P$ as  tuples.  Let $\cB^*$  be the 
set $$ \cB^* = \{p[i] \in P \mid p\in v_i\cB\}$$
 where $p[i]=(0,\dots,0,p,0,\dots,0)$ with
$p$ in the $i$-th component. Note that  $\cB^*$ is a $K$-basis of $P$.   We  extend the admissible  order to  $\cB^*$. For this we order the set $\cI$ 
and set $p[i] \succ q[j]$ if $p \succ q $ or $p=q$ and $i >j$.

\begin{defn}
1) Let $x = \sum_k \alpha_k b_k \in P$ where $\alpha_k \in K$ and $b_k \in \cB^*$. Set 
$$ \tip(x) = b_m \mbox{ where $\alpha_m \neq 0$ and if  $\alpha_l \neq 0$ and $l \neq m$,  then } b_m \succ b_l. $$

2) If $\{ 0 \} \neq X \subset P$ then 
$$ \tip (X) = \{ \tip (x) \mid x \in X \}.$$
\end{defn} 

3) We call an element $x \in P$ {\it right uniform}, if there exists a vertex $v \in Q_0$ such that $xv = x$.

4) We say  $ X \subset P$ is {\it tip-reduced} if for $x, x' \in X$, $x \neq x'$, and $\tip(x) = p[i]$ and  $\tip(x')=p'[i]$, then $p$ is not a prefix of $p'$.  

If an element $x \in P$ is right uniform then $xKQ \simeq \mathfrak{t} (x) KQ$ where
$ \mathfrak{t} (x)$ is the vertex $v \in KQ$ such that $xv =x$.

\begin{proposition}\cite{G}
Let $X$ be a tip-reduced right uniform generating set of  a  $KQ$-submodule $L$ of a projective $KQ$-module $P$. Then $$L = \bigoplus_{x \in X} x KQ.$$
\end{proposition}

 Next, we provide the general set-up for the construction of an order resolution
as found in \cite{GS}.

Let $M$ be a $KQ/I$-module which is finitely presented as a $KQ$-module and let $$0 \to L \to P \to M \to 0$$ be a finite $KQ$-presentation of $M$. Then we have the following commutative exact diagram. 

$$
\xymatrix{
&0\ar[d]&0\ar[d]\\
0\ar[r]&PI\ar[r]^{=}\ar[d]&PI\ar[d]\ar[r]&0\ar[d]\\
0\ar[r]&L\ar[r]\ar[d]&P\ar[d]\ar[r]&M\ar[d]^{=}\ar[r]&0\\
0\ar[r]&\Omega_{\Lambda}(M)\ar[r]\ar[d]&P/PI\ar[r]\ar[d]&M\ar[r]\ar[d]&0\\
&0&0&0}
$$

The construction of the projective resolution of $M$ is based on the construction of a sequence of sets, which we now introduce. 
Let $\cF^0 = \{ v_i[i] \in P \mid  i \in \cI  \}$.   
By \cite{G}, there is a right uniform  tip-reduced generating set $\widehat{\cF}^1$ of $L$. Define  $\cF^1 = \{ f^1 \in \widehat{\cF}^1 \mid f^1 \notin 
PI \} $.  
 Note that $\cF^1 \subset \bigoplus_{f^0 \in \cF^0} f^0 KQ$. Next, assume we have constructed
$\cF^{n-1}$ and $\cF^{n-2}$.  From this data, one constructs the tip-reduced set $\cF^n$
such that if $f^n\in \cF^n$ then $f^n$ is right uniform and  $f^n\in \bigoplus_{f^{n-1}\in\cF^{n-1}} f^{n-1}KQ$. 

We note that if $f\in \cF^n$ then $\mathfrak{t}(f)=\mathfrak{t}(\tip(f))$  by uniformity.
We briefly recall some facts about the sets $\cF^n$.  The tip set of $\cF^n$
is determined by the tip sets of $\cF^{n-1}$, $\cF^{n-2}$, and 
 the reduced \grb basis of $I$.  The point is that
if the tip sets of $\cF^{n-1}$ and $\cF^{n-2}$ coincide for the order resolutions  
of a $\Lambda$-module $M$ and $\Lambda'$-module $M$, then the
tip set of the set $\cF^{n}$ for the two modules will coincide \cite{GS}.

The $n$-th projective $\cP^n$ in the $KQ/I$-order resolution of $M$ is given by 
$$ \cP^n = \bigoplus_{f \in \cF^n} \mathfrak{t} (f) KQ/I.$$

\begin{proposition}\label{really awful} Suppose $S$, respectively $S'$, is a simple one dimensional 
$KQ/I$-module, respectively $KQ/I_{Mon}$-module, associated to a vertex $v\in Q$.
Let $\cP$ be an order resolution of $S$ with generating sets $\cF^n$ and $\cP'$ be an order resolution of $S'$  with generating sets $(\cF')^{n}$. Then for all $n$,  $\tip(\cF^n) = \tip((\cF')^n)$. 
\end{proposition}

\begin{proof}
$S$ and $S'$ have the same $KQ$-presentation, namely 
$$ 0 \to \bigoplus_{\mathfrak{o} (a) = v} a KQ \to vKQ \to S \to 0$$ 
and 
$$ 0 \to \bigoplus_{\mathfrak{o} (a) = v} a KQ \to vKQ \to S' \to 0.$$ 
Furthermore,   if $\cG$ and $\cG'$ are the \grb bases for $I$ and $I_{Mon}$ respectively   then by definition $ \tip(\cG) = \tip(\cG')$.   Thus,  $\tip(\cF^0) = \tip((\cF')^0)$,  $\tip(\cF^1) = \tip((\cF')^1)$. By our earlier discussion, we conclude  $\tip(\cF^2) = \tip((\cF')^2)$.  Continuing, we see
that  $\tip(\cF^n) = \tip((\cF')^n)$, for $n\ge 0$.
 
\end{proof}

Recall that  the {\it $n$th-Betti number of a projective resolution  $\cP$ of a  $\Lambda$-module $M$} is the sequence $(a_1, \ldots, a_{\vert Q_0 \vert})$ where $a_i$ is the number of  $f \in  \cF^n$ such that  $ \mathfrak{t} (f) = v_i$. 
 
The proof of the following Theorem follows directly from Proposition~\ref{really awful}

\begin{theorem}
Let $S$  be a simple $KQ/I$-module and $S'$ a simple $KQ/I_{Mon}$-module, both corresponding to the same vertex $v$ in $Q_0$. Let $\cP$ and $\cP'$ be the order resolutions of $S$ and $S'$ respectively. Then the Betti numbers of $\cP$ and $\cP'$ coincide. 
\end{theorem}
 
 \begin{corollary}\label{cor-Betti numbers}
 Let $KQ/I$ and $KQ/I'$ be such that $I_{Mon} = I'_{Mon}$. Let $S$  be a simple $KQ/I$-module and $S'$ a simple $KQ/I'$-module, both corresponding to the same vertex $v$ in $Q_0$. Let $\cP$ and $\cP'$ be the order resolutions of $S$ and $S'$ respectively. Then the Betti numbers of $\cP$ and $\cP'$ coincide. 
 \end{corollary}

For a monomial algebra $KQ/I_{Mon}$ it is easy to see, using \cite{GHZ}, that  the order resolution of the simple $KQ/I_{Mon}$-modules is minimal in the sense that the image of the differential from the $n$th-projective
module to the $n-1$st projective module  $P^{n-1}$ is contained in $P^{n-1}J$, where $J$ is the ideal in $KQ$
generated by the arrows of $KQ$.

\bibliographystyle{plain}

 \end{document}